\documentclass{amsart}
\usepackage{latexsym,amssymb,amsfonts,amsmath,graphicx,fullpage,url, amsthm, paralist}
\usepackage[vcentermath,enableskew]{youngtab}
\usepackage{color}
\usepackage[all]{xy}
\usepackage{tikz,bbm}
\usepackage{verbatim}
\usepackage{tikz-cd}
\usepackage{latexsym,amssymb,amsfonts,amsmath,graphicx,fullpage,url}
\usepackage[vcentermath,enableskew]{youngtab}
\usepackage{color}
\usepackage[all]{xy}
\usepackage{verbatim}
\usepackage{tikz,bbm}
\usepackage{caption}
\usepackage{subcaption}
\usepackage{hyperref}
\usepackage{bm}
\usepackage{colonequals}
\usepackage{mathtools}
\definecolor{darkgreen}{rgb}{0.0, 0.5, 0.0}

\usepackage[alphabetic]{amsrefs}

\newtheorem*{theorem1*}{Theorem \ref{thm:sigmaphi-sigmapsi-lorentzian}}
\newtheorem*{question1*}{Question \ref{q:1}}

\newtheorem{theorem}[subsection]{Theorem}
\newtheorem{proposition}[subsection]{Proposition}

\newtheorem{lemma}[subsection]{Lemma}
\newtheorem{corollary}[subsection]{Corollary}
\theoremstyle{definition}
\newtheorem{definition}[subsection]{Definition}
\newtheorem{question}[subsection]{Question}
\newtheorem{example}[subsection]{Example}

\theoremstyle{remark}
\newtheorem{remark}[subsection]{Remark}
\numberwithin{equation}{section}
\def\F{\mathcal{F}}
\def\vol{{\rm vol}}
\def\F{{\mathcal{F}}}

\def\R{\mathbb{R}}

\newcommand{\be}{\begin{equation}}
\newcommand{\ee}{\end{equation}}

\newcommand{\B}{\bigg}

\newcommand{\ef}[1]{{\sf ef}(#1)}

\newcommand{\NN}{\mathbb N}
\newcommand{\ZZ}{\mathbb Z}
\newcommand{\RR}{\mathbb R}

\newcommand{\bd}{\begin{definition}}
\newcommand{\ed}{\end{definition}}
\newcommand{\bt}{\begin{theorem}}
\newcommand{\et}{\end{theorem}}
\newcommand{\bl}{\begin{lemma}}
\newcommand{\el}{\end{lemma}}
\newcommand{\bp}{\begin{proposition}}
\newcommand{\ep}{\end{proposition}}
\newcommand{\bc}{\begin{corollary}}
\newcommand{\ec}{\end{corollary}}

\font\co=lcircle10

\def\jr{\smash{\raise2pt\hbox{\co \rlap{\rlap{\char'005} \char'007}}
               \raise6pt\hbox{\rlap{\vrule height6.5pt}}
               \raise2pt\hbox{\rlap{\hskip4pt \vrule height0.4pt depth0pt
                width7.7pt}}}}
\def\je{\smash{\raise2pt\hbox{\co \rlap{\rlap{\char'005}
                \phantom{\char'007}}}\raise6pt\hbox{\rlap{\vrule height6pt}}}}
\def\+{\smash{\lower2pt\hbox{\rlap{\vrule height14pt}}
                \raise2pt\hbox{\rlap{\hskip-3pt \vrule height.4pt depth0pt
                width14.7pt}}}}

\def\textcross{\ \smash{\lower4pt\hbox{\rlap{\hskip4.15pt\vrule height14pt}}
                \raise2.8pt\hbox{\rlap{\hskip-3pt \vrule height.4pt depth0pt
                width14.7pt}}}\hskip12.7pt}
\def\textelbow{\ \hskip.1pt\smash{\raise2.8pt%
                \hbox{\co \hskip 4.15pt\rlap{\rlap{\char'005} \char'007}
                \lower6.8pt\rlap{\vrule height3.5pt}
                \raisebox{?}{?}3.6pt\rlap{\vrule height3.5pt}}
                \raise2.8pt\hbox{%
                  \rlap{\hskip-7.15pt \vrule height.4pt depth0pt width3.5pt}%
                  \rlap{\hskip4.05pt \vrule height.4pt depth0pt width3.5pt}}}
                \hskip8.7pt}

\newcommand{\old}[1]{}

\def\ee{{\bf e}}
  
   \def\vol{{\rm vol}}

 \def\f_H{{\bf w}}
 \def\f{{\bf f}}
 \def\a{{\bf a}}

\def\R{\mathbb{R}}

\def\Z{\mathbb{Z}}

 \def\F{\mathcal{F}}

\def\ee{{\bf e}}
  
   \def\vol{{\rm vol}}

 \def\f_H{{\bf w}}
 \def\f{{\bf f}}
 \def\a{{\bf a}}
\def\R{\mathbb{R}}

\def\Z{\mathbb{Z}}

 \def\F{\mathcal{F}}

\def\ee{{\bf e}}
  
   \def\vol{{\rm vol}}

 \def\f_H{{\bf w}}
 \def\f{{\bf f}}
 \def\a{{\bf a}}

\def\R{\mathbb{R}}

\def\Z{\mathbb{Z}}

 \def\F{\mathcal{F}}

\newcommand{\eee}{\end{equation}}

\newcommand\multiset[2]%
{\mathchoice{\left(\kern-0.5em{\binom{#1}{#2}}\kern-0.5em\right)}
            {\bigl(\kern-0.3em{\binom{#1}{#2}}\kern-0.3em\bigr)}
            {\bigl(\kern-0.3em{\binom{#1}{#2}}\kern-0.3em\bigr)}
            {\bigl(\kern-0.3em{\binom{#1}{#2}}\kern-0.3em\bigr)}}

\def\R{\mathbb{R}}

\title{Lorentzian polynomials from polytope projections}

\author{Karola M\'esz\'aros}
\author{Linus Setiabrata}
\thanks{Karola M\'esz\'aros received support from NSF Grant DMS-1501059, CAREER NSF Grant DMS-1847284
and  a von Neumann Fellowship funded by the Friends of the Institute for Advanced Study.}

\address{Department of Mathematics, Cornell University, Ithaca, NY 14853}
 
\email{karola@math.cornell.edu, ls823@cornell.edu}

\date{\today}

\begin{document}

\maketitle

\begin{abstract}  Lorentzian polynomials, recently introduced by Br\"and\'en and Huh, generalize the notion of log-concavity of sequences to homogeneous polynomials whose supports are integer points of generalized permutahedra.  Br\"and\'en and Huh show that  normalizations of  polynomials equaling integer point transforms of generalized permutahedra  are Lorentzian; moreover, normalizations of certain projections of  integer point transforms  of generalized permutahedra with zero-one vertices  are also Lorentzian. Taking this polytopal perspective further, we show that normalizations of certain projections of integer point transforms of flow polytopes  (which, before projection, are not Lorentzian), are also Lorentzian.  \end{abstract}

\section{Introduction} 

Log-concavity of sequences is a classical notion, which often is either very easy or notoriously difficult to prove.  A sequence $a_0, a_1, \ldots, a_n$ is said to be log-concave if $a_i^2\geq a_{i-1}a_{i+1}$ for $i \in [n-1]$.  In groundbreaking recent work Br\"and\'en and Huh \cite{bh2019} introduced   \textit{Lorentzian polynomials} (see Section \ref{sec:lor} for definition), which generalize the  notion of log-concavity. Just one of their theory's many consequences are the celebrated Alexandrov-Fenchel inequalities on mixed volumes of Minkowski sums of polytopes; these inequalities follow from the volume polynomial being Lorentzian \cite[Theorem 9.1]{bh2019}. 

\medskip

Our motivation for the present paper is simple: (whenever possible) understand Lorentzian polynomials  \textit{polytopally}.   Br\"and\'en and Huh show that the support of a Lorentzian polynomial form the integer points of a \textit{generalized permutahedron}, a beautiful polytope studied extensively by Postnikov in \cite{postnikov2009}.   \medskip

 Recall that for a polytope $P \subset \mathbb{R}^n$, the  \textbf{integer point transform} of $P$  is defined as  \begin{equation} \sigma_P(x_1, \ldots, x_n)=\sum_{{\bf p} \in P\cap \mathbb{Z}^n}{\bf x}^{\bf p},  \text{ where ${\bf x}^{\bf p}=\prod_{i=1}^n x_i^{p_i}$.} \end{equation} 
 
 Define the \textbf{normalization operator} $N$ on $\mathbb{R}[x_1,\ldots,x_n]$ by
\begin{equation}\label{eq:norm}
N(\mathbf{x}^{\bf \alpha})=\frac{\mathbf{x}^{\bf \alpha}}{{\bf \alpha}!},
\end{equation}
where for a vector ${\bf \alpha} = (\alpha_1, \dots, \alpha_n)$ of nonnegative integers we write $\alpha!$ to mean $\prod_{i=1}^n\alpha_i!$.
\medskip

Br\"and\'en and Huh  show that the normalization of any polynomial equaling the integer point transform of a generalized permutahedron is always a Lorentzian polynomial \cite[Theorem 7.1 (4) \& (7)]{bh2019}. It also follows from \cite[Theorem 2.10]{bh2019} and \cite[Corollary 6.7]{bh2019} that (normalizations of) certain projections of  integer point transforms  of generalized permutahedra with 01 vertices are also Lorentzian. In joint work with Huh, Matherne and St.~Dizier the first author has  shown that the normalization of  certain projections of  the integer point transforms of  Gelfand-Tsetlin polytopes are Lorentzian \cite[Theorem 1]{hmms2019}. The question lurking in the background of  the present work is:

 \begin{question} \label{q:1} {\it What conditions on the polytope/projection pair ascertain that the normalization of the  projection of the integer point transform of the polytope is Lorentzian?}  \end{question}
 
 \medskip

The current paper adds a natural class of polytope/projection pairs yielding Lorentzian polynomials: flow polytopes/projection onto a coordinate hypersurface.    

\medskip

The  \textbf{flow polytope} $\mathcal{F}_G({\bf a})$ associated to    a loopless graph $G$ on vertex set $[n+1]$ with edges directed from smaller to larger vertices and to the \textbf{netflow vector} ${\bf a}=(a_1,\,\ldots, a_{n+1}) \in \Z^{n+1}$ is:
	\begin{equation}\label{def:fp}\mathcal{F}_G({\bf a})= \{f\in \R^{E(G)}_{\geq 0}\colon\, M_Gf={\bf a} \}, \end{equation} where $M_G$  is the incidence matrix of $G$; that is,   the columns of $M_G$ are the vectors $e_i-e_j$ for $(i,j)\in E(G)$, $i<j$,  where $e_i$ is the $i$-th standard basis vector in $\R^{n+1}$. 
	The points $f \in \mathcal{F}_G({\bf a})$ are called (${\bf a}$-)\textbf{flows} (on $G$).

\vspace{.05in}
Observe that the number of integer points in $\mathcal{F}_G({\bf a})$ is  the number of ways to write ${\bf a}$ as a nonnegative integral combination of the vectors $e_i-e_j$ for edges $(i,j)$ in $G$, $i<j$, which we refer to as the \textbf{Kostant partition function} $K_G({\bf a})$.

We define two natural projections $\varphi$ and $\psi$ of $\mathcal{F}_G({\bf a})$ onto generalized permutahedra in Propositions \ref{prop:flow-project-gp} and \ref{prop:gp-project-smallgp} in Section \ref{sec:projections-of-flow-polytopes}. 
The projections $\varphi$ and $\psi$ induce projections on the integer point transform $\sigma_{\mathcal F_G(\mathbf a)}({\bf x})$ of $\mathcal F_G(\mathbf a)$, acting on monomials via $\mathbf x^f \mapsto\mathbf x^{\varphi(f)}$ and $\mathbf x^f\mapsto\mathbf x^{\psi(f)}$. The resulting projected polynomials are denoted
\begin{equation} \label{def:1}
\sigma_{G(\mathbf a)}^\varphi(\mathbf x) \overset{\rm def}=\sum_{\mathbf p \in \mathcal F_G(\mathbf a) \cap \mathbb{Z}^{|E(G)|}} \mathbf x^{\varphi(\mathbf p)},
\end{equation} 
and
\begin{equation} \label{def:2}
\sigma_{G(\mathbf a)}^\psi(\mathbf x) \overset{\rm def}=\sum_{\mathbf p  \in \mathcal F_G(\mathbf a) \cap \mathbb{Z}^{|E(G)|}} \mathbf x^{\psi(\mathbf p)}.
\end{equation}

\medskip

While the normalization of the integer point transform of $\mathcal F_G(\mathbf a)$ is not Lorentzian in general, we prove that the normalizations of its projections $\sigma_{G(\mathbf a)}^\varphi$ and $\sigma_{G(\mathbf a)}^\psi$ are always Lorentzian:
 
\begin{theorem1*} 
Let $G$ be a loopless  directed graph on the vertex set  $[n+1]$ with a unique sink, and let $\mathbf a = (a_1, \dots, a_{n+1}) \in \Z_{\geq 0}^{n}\times \Z_{\leq 0}$. The polynomials $N(\sigma_{G(\mathbf a)}^\varphi)$ and $N(\sigma_{G(\mathbf a)}^\psi)$ are Lorentzian.
\end{theorem1*}

  Theorem \ref{thm:sigmaphi-sigmapsi-lorentzian} implies  that the Kostant partition function is log-concave along root directions (Corollary~\ref{cor:kostant-log-concave}).  We remark that  log-concavity of the Kostant partition function   along root directions is also a corollary of volume polynomials (of flow polytopes) being Lorentzian (Theorem~\ref{lem:vol-is-lorentzian}).

\bigskip

\noindent {\bf Roadmap of the paper.}   Section \ref{sec:back} contains the necessary background on Lorentzian polynomials, generalized permutahedra and flow polytopes. Section \ref{sec:projections-of-flow-polytopes} introduces the projections $\varphi$ and $\psi$ of $\mathcal{F}_G({\bf a})$ onto generalized permutahedra that we are interested in, while Section \ref{subsec:the-fibers-of-our-projection} studies their fibers. Section \ref{sec:?} establishes our main result, Theorem \ref{thm:sigmaphi-sigmapsi-lorentzian}. Section \ref{sec:proj} prods Question \ref{q:1}. \section{Background} 
\label{sec:back}

In this section we give background on the main players of the paper: Lorentzian polynomials, generalized permutahedra and flow polytopes. 

\subsection{Lorentzian polynomials and generalized permutahedra} \label{sec:lor}
Let $\mathbb{N}=\{0,1, 2, \ldots\}$, and denote by $e_i$ the $i$th standard basis vector of $\mathbb{N}^n$. A subset $J\subseteq \mathbb{N}^n$ is called \textbf{$\mathrm{M}$-convex} if for any index $i$ and any $\alpha,\beta \in J$ whose $i$th coordinates satisfy $\alpha_i > \beta_i$,
there is an index $j$ satisfying
\[\alpha_j<\beta_j, \ \  \alpha-e_i+e_j \in J, \ \ \text{and} \ \ \beta-e_j+e_i \in J.\]

The convex hull of an $\mathrm{M}$-convex set is a polytope also called a \textbf{generalized permutahedron}. A special class of generalized permutahedra consist of Minkowski sums of scaled \textbf{coordinate simplices}: for a subset $S\subseteq [n]$, the coordinate simplex $\Delta_S\subseteq \RR^n$ is the convex hull of the coordinate basis vectors $\{e_i\}_{i \in S}$. Minkowski sums of scaled coordinate simplices are called \textbf{$y$-generalized permutahedra}.

Let $\mathrm{H}^d_n$ be the space of degree $d$ homogeneous polynomials with real coefficients in the $n$ variables $x_1,\ldots,x_n$. For $f \in \mathrm{H}^d_n$, we write $\text{supp}(f) \subseteq \mathbb{N}^n$ for the support of $f$. For $f\in \mathrm{H}^d_n$, denote by $\frac{\partial}{\partial x_{i}} f$ the partial derivative of $f$ relative to $x_i$.
The \textbf{Hessian} of a homogenous quadratic polynomial $f \in \mathrm H^2_n$ is the symmetric $n\times n$ matrix $H = (H_{ij})_{i,j\in[n]}$ defined by $H_{ij} = \partial_i\partial_j f$.
	The set $\mathrm{L}^d_n$ of \textbf{Lorentzian polynomials} with degree $d$ in $n$ variables is defined as follows.
	Set $\mathrm L_n^1\subseteq \mathrm H^1_n$ to be the set of all linear polynomials with nonnegative coefficients. Let $\mathrm{L}^2_n \subseteq \mathrm{H}^2_n$ be the subset of quadratic polynomials with nonnegative coefficients whose Hessians have at most one positive eigenvalue and which have $\mathrm{M}$-convex support. For $d>2$, define $\mathrm{L}^d_n \subseteq \mathrm{H}^d_n$ recursively by 
	\[\mathrm{L}^d_n=\left\{f \in \mathrm{M}^d_n\colon \frac{\partial}{\partial x_{i}} f \in \mathrm{L}^{d-1}_n \mbox{ for all } i\right\}.\]
	where $\mathrm{M}^d_n \subseteq \mathrm{H}^d_n$ is the set of polynomials with nonnegative coefficients whose supports are $\mathrm{M}$-convex.

Since $f \in \mathrm{M}^d_n$ implies $\frac{\partial}{\partial x_{i}} f \in \mathrm{M}^{d-1}_n$, we have
\[
\mathrm{L}^d_n=\left\{f \in \mathrm{M}^d_n\colon \frac{\partial}{\partial x_{i_1}}\frac{\partial}{\partial x_{i_2}} \cdots \frac{\partial}{\partial x_{i_{d-2}}} f \in \mathrm{L}^{2}_n \mbox{ for all }i_1,i_2,\ldots,i_{d-2}\in [n] \right\}.
\]

Recall  the \textbf{normalization operator} $N$ on $\mathbb{R}[x_1,\ldots,x_n]$:
$$N(\mathbf{x}^{\alpha})=\frac{\mathbf{x}^\alpha}{\alpha!},$$
where for a vector $\alpha = (\alpha_1, \dots, \alpha_n)$ of nonnegative integers we write $\alpha!$ to mean $\prod_{i=1}^n\alpha_i!$.  

For a quadratic polynomial 
\begin{equation*}
f(\mathbf x) = \sum_{1\leq i \leq j \leq n}c_{ij}x_ix_j \in \mathrm M_n^2,
\end{equation*} 
observe that the $ij$-th entry of the Hessian of $N(f)$, namely the quantity $\partial_i\partial_jN(f)$, is the coefficient $c_{ij}$ of $x_ix_j$ in $f$. Thus, asking whether $N(f)$ is Lorentzian, equivalently whether the Hessian of $N(f)$ has at most one positive eigenvalue, can be phrased purely in terms of the coefficients of $f$. For arbitrary polynomials $f \in \mathrm M_n^d$ we use the following lemma:

\begin{lemma}
\label{lem:NDN-is-coeff}
The linear operator $N^{-1}\frac\partial{\partial x_i}N$ acts on polynomials by
\begin{equation}
\label{eq:NDN-is-coeff}
\B(N^{-1}\frac{\partial}{\partial x_i}N\B) \colon \sum_\alpha c_\alpha \mathbf x^\alpha \mapsto \sum_{\alpha\colon \alpha_i \geq 1}c_\alpha\mathbf x^{\alpha - e_i}.
\end{equation}
\end{lemma}


We arrive at the following criterion for Lorentzian polynomials. It is our main workhorse in the proof of Theorem~\ref{thm:sigmaphi-sigmapsi-lorentzian}.

\begin{lemma}
\label{lem:lorentzian-criterion}
Let $f$ be a homogeneous polynomial of degree $d\geq 2$, and suppose
\[
f(\mathbf x) = \sum_\alpha c_\alpha\mathbf x^\alpha \in \mathrm M_n^d.
\] 
For each $\mathbf d = (d_1, \dots, d_n)$ with $d_1 + \dots + d_n = d-2$ and $d_i \in \mathbb{Z}_{\geq 0}$ for $i \in [n]$, define the $n\times n$ matrix
\[
H_{\mathbf d} = (H_{ij;\mathbf d})_{i,j\in[n]}; \hspace{1cm} H_{ij;\mathbf d} = c_{\mathbf d + e_i + e_j}
\]
consisting of coefficients of $f$. Then $N(f)\in \mathrm L_n^d$ if and only if $H_{\mathbf d}$ has at most one positive eigenvalue for each $\mathbf d$.
\end{lemma}

\begin{proof}
Note that normalization and differentiation preserve $M$-convexity of the support of a polynomial. By Lemma~\ref{lem:NDN-is-coeff}, we obtain
\[
N^{-1}\partial_{\mathbf d}N(f) = \sum_{\alpha\colon \alpha \geq \mathbf d} c_\alpha \mathbf x^{\alpha - \mathbf d}\in \mathrm M_n^2.
\]  
Because the Hessian of $N(N^{-1}\partial_{\mathbf d}N(f)) = \partial_{\mathbf d}N(f)$ is $H_{\mathbf d}$, by definition $N(f)$ is Lorentzian if and only if $H_{\mathbf d}$ has at most one positive eigenvalue for each $\mathbf d$.
\end{proof}

The coefficients of Lorentzian polynomials satisfy a log-concavity inequality as in Proposition \ref{lem:log-concavity-of-coeffs} below. It is in this sense that Lorentzian polynomials generalize the notion of log-concavity.

\begin{proposition}[{\cite[Proposition 9.4]{bh2019}}]
\label{lem:log-concavity-of-coeffs}
If $f(\mathbf x) = \sum_\alpha c_\alpha \mathbf x^\alpha$ is a homogeneous polynomial on $n$ variables so that $N(f)$ is Lorentzian, then for any $\alpha\in\NN^n$ and any $i,j \in [n]$ the inequality
\[
c_\alpha^2 \geq c_{\alpha + e_i - e_j} c_{\alpha - e_i + e_j}
\]
holds.
\end{proposition}

This proposition can be seen  as a consequence of Cauchy's Interlacing Theorem. We recall below a special case of Cauchy's Interlacing Theorem, which we will use later.

\begin{proposition}[Cauchy's Interlacing Theorem, {\cite[Theorem 10.1.1]{parlett1998}}] 
\label{prop:cauchy-interlacing}
Let $A$ be a symmetric $n \times n$ matrix, and let $S\subseteq[n]$, and $m = |S|$. Let $B = A_S$ be thde $m\times m$ principal submatrix of $A$ given by $B = (a_{ij})_{i,j\in S}$. Let $\alpha_1 \leq \dots \leq \alpha_n$ be the eigenvalues of $A$ and let $\beta_1 \leq \dots \leq \beta_m$ be the eigenvalues of $B$. Then for every $j \in [m]$,
\[
\alpha_j \leq \beta_j \leq \alpha_{n-m+j}.
\]
In other words, the $j$th smallest eigenvalue of $A$ is at most the $j$th smallest eigenvalue of $B$, and the $j$th largest eigenvalue of $A$ is at least the $j$th largest eigenvalue of $B$.
\end{proposition}

We recall two important theorems about Lorentzian polynomials here:

\begin{theorem}[{\cite[Theorem 2.10]{bh2019}}]
\label{lem:f(Av)}
If $f\in\mathrm L_n^d$ is a Lorentzian polynomial in $n$ variables, and $A$ is an $n\times m$ matrix with nonnegative entries, then $f(A\mathbf v)\in \mathrm L_m^d$ is a Lorentzian polynomial in the $m$ variables $\mathbf v = (v_1, \dots, v_m)$.
\end{theorem}

\begin{theorem}[{\rm \cite[Theorem 9.1]{bh2019}}]
\label{lem:vol-is-lorentzian}
Let $K = (K_1, \dots, K_n)$ be convex bodies in $\RR^d$. The volume polynomial
\[
(w_1, \dots, w_n) \mapsto \vol(w_1K_1 + \dots + w_nK_n)
\]
is a Lorentzian polynomial.
\end{theorem}

 \subsection{Flow polytopes.}  
	Recall the definition of flow polytopes in \eqref{def:fp}. We record several properties of them here which we will be using in later sections. 
	
	\begin{lemma}[\cite{schrijver2003}]
\label{lem:vertices-flow}
For any graph $G$ on the vertex set $[n+1]$, the vertices of the flow polytope $\mathcal F_G(e_1 - e_{n+1})$ are unit flows with support equal to {{\sf p}}, where {{\sf p}} is an increasing path from vertex $1$ to vertex $n+1$.
\end{lemma}

\begin{proposition}[{\cite[Section 3.4]{bv2008}}] \label{prop:flow-minkowski}
For nonnegative integers $a_1,\ldots,a_n$ and $G$ a graph on the vertex set $[n+1]$ we have that
\begin{equation} \label{eq:flow-minkowski}
\F_G(\a) = a_1 \F_G(e_1-e_{n+1}) + a_2 \F_G(e_2-e_{n+1}) + \cdots +
a_n \F_G(e_n-e_{n+1}).
\end{equation}
\end{proposition}

	  \begin{theorem}[Baldoni--Vergne volume formula, {\cite[Theorem 38]{bv2008}}]
\label{thm:baldoni-vergne} 
Let $G$ be a directed graph on the vertex set  $[n+1]$ with a unique sink, so that edges are oriented from a smaller vertex to a larger vertex. Let $\mathbf x = (x_1, \dots, x_n, -\sum_{i=1}^nx_i), x_i \in \ZZ_{\geq 0}$. Then
\[
\vol\,\mathcal F_G(\mathbf x) = \sum_{\mathbf j}K_G(j_1 - \textup{out}_1, \dots, j_n - \textup{out}_n, 0)\frac{\mathbf x^{\mathbf j}}{\mathbf j!},
\]
for $\textup{out}_i = \textup{outdeg}_i - 1$, where $\textup{outdeg}_i$ denotes the outdegree of vertex $i$ in $G$. The sum is over weak compositions $\mathbf j = (j_1, \dots, j_n)$ of $|E(G)| - n$ that dominate $(\textup{out}_1, \dots, \textup{out}_n)$, that is, for every $i\in[n]$ we have
\[
j_1 + \dots + j_i \geq \textup{out}_1 + \dots + \textup{out}_i.
\]
\end{theorem}
In the above $\mathbf x^{\mathbf j}=\prod_{i=1}^n x_i^{j_i}$ and  $\mathbf j!=\prod_{i=1}^n j_i!$.

\section{Projections of flow polytopes onto generalized permutahedra}
\label{sec:projections-of-flow-polytopes}

In this section, we define the projections  $\varphi\colon\mathcal F_G(\mathbf a)\to\mathcal P(G;\mathbf a)$ and $\psi\colon\mathcal F_G(\mathbf a)\to\mathcal Q(G;\mathbf a)$, where $\mathcal P(G;\mathbf a)$ and  $\mathcal Q(G;\mathbf a)$ are $y$-generalized permutahedra   (see Propositions~\ref{prop:flow-project-gp} and~\ref{prop:gp-project-smallgp}). We study their  fibers in Section~\ref{subsec:the-fibers-of-our-projection}, leading us to  explicit expressions for  the polynomials $\sigma_{G(\mathbf a)}^\varphi$ and $\sigma_{G(\mathbf a)}^\psi$; see Corollary~\ref{cor:sigmaphi-sigmapsi-coeffs}. In Section \ref{sec:?} we use these expressions to prove Theorem \ref{thm:sigmaphi-sigmapsi-lorentzian}. 
\medskip

\noindent \textbf{Notational Conventions for Sections \ref{sec:projections-of-flow-polytopes}, \ref{subsec:the-fibers-of-our-projection} and \ref{sec:?}.} Unless specified otherwise, $G$ denotes a loopless directed graph on the vertex set  $[n+1]$ vertices  with a unique sink. Every edge of $G$ is oriented from its smaller vertex to its larger vertex. All flow polytopes $\mathcal F_G(\mathbf a)$ have netflow vector $\mathbf a \in \mathbb{Z}_{\geq 0}^n\times \mathbb{Z}_{\leq 0}$. For a finite set $S$, we denote by $\RR^S$ the real vector space consisting of $\RR$-linear combinations of elements in $S$; observe that for sets $S\subseteq T$, the vector space $\RR^S$ canonically embeds in $\RR^T$ as a coordinate hyperspace. We write $\RR^n$ to denote $\RR^{[n]}$.

\medskip

\begin{definition}
For $i,j \in V(G) = [n+1]$, we denote by $M(i,j)\in \NN_{\geq 0}$ the number of edges from $i$ to $j$, and by $\{(i,j;k)\}_{k \in [M(i,j)]}\subseteq E(G)$ the set of edges of $G$ connecting $i$ to $j$. 
\end{definition}

\begin{definition}[see Example~\ref{ex:STst}]
We denote by $S_G$ the set of all edges incident to the sink, that is,
\[
S_G \overset{\rm def}= \{e \in E(G)\colon e = (i,n+1;k) \text{ for some } i \in[n], k \in [M(i,n+1)]\}.
\]
For $i \in [n]$, let $S_{G,i}\subseteq S_G$ be the set of edges incident to $n+1$ which can be reached from vertex $i$, that is, if $\overline G$ denotes the transitive closure of $G$, then
\[
S_{G,i}\stackrel{\rm def}= \{e \in S_G\colon e = (j,n+1;k) \text{ and } (i,j) \in E(\overline G)\}.
\]
Denote by $T_G$ the set of all vertices incident to the sink, that is, 
\[
T_G \stackrel{\rm def}= \{i \in V(G) \colon M(i,n+1) \geq 1\}.
\]
For $i \in [n]$, let $T_{G,i}\subseteq T_G$ be the set of vertices adjacent to $n+1$ which can be reached from vertex $i$, that is,
\[
T_{G,i}\stackrel{\rm def}= \{j \in T_G\colon (i,j) \in E(\overline G)\}.
\]
\end{definition}

\begin{example}
\label{ex:STst}
Let $G$ be as in Figure~\ref{fig:STst}. The set $S_G\subseteq E(G)$ consists of the  blue edges, while $S_{G,2}$ consists of the four  blue edges emanating from vertices 2 and 4. If $G'$ denotes the graph obtained from $G$ by removing the edge $(2,3;1) \in E(G)$, then $S_{G',2}$ would only consist of the two  blue edges emanating from vertex 2.

The set $T_G\subseteq V(G) = [5]$ is equal to $\{1,2,4\}$, and $T_{G,3} = \{4\}$.
\begin{figure}[ht]
\begin{center}
\includegraphics[scale=0.6]{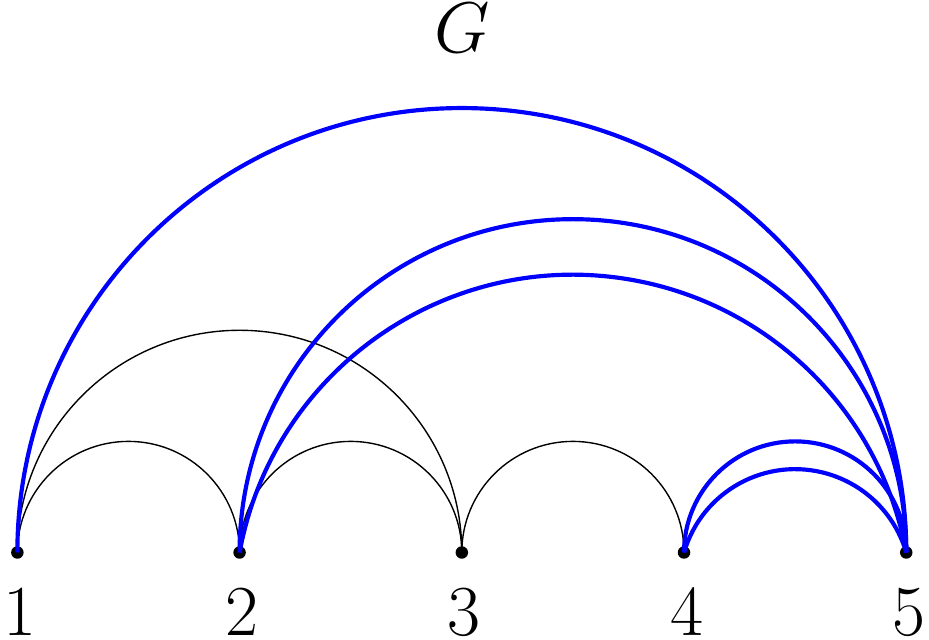}
\end{center}
\caption{A graph $G$ satisfying the conventions of this section, with edge orientations suppressed.}
\label{fig:STst}
\end{figure}
\end{example}

\begin{proposition}
\label{prop:flow-project-gp}
There is a projection
\[
\varphi\colon \mathcal F_G(\mathbf a)\twoheadrightarrow \mathcal P(G;\mathbf a),
\]
where $\mathcal P(G;\mathbf a)\subseteq\RR^{S_G}$ is the $y$-generalized permutahedron defined by
\[
\mathcal P(G;\mathbf a)\stackrel{\rm def}=\sum_{i\in[n]}a_i\Delta_{S_{G,i}}.
\] 
The map $\varphi$ is given by projecting a flow in $\mathcal F_G(\mathbf a)$ to the coordinates corresponding to edges in $S_G$.
\end{proposition} 
\begin{proof}
Proposition~\ref{prop:flow-minkowski} asserts that
\[
\mathcal F_G({\bf a}) = \sum_{i\in[n]}a_i\mathcal F_G(e_i - e_n).
\]
Because linear maps factor through Minkowski sums, we obtain
\[
\varphi(\mathcal F_G({\bf a})) = \sum_{i \in [n]}a_i\varphi(\mathcal F_G(e_i - e_{n+1})).
\]
Observe that $\varphi(\mathcal F_G(e_i - e_{n+1})) = \Delta_{S_{G,i}}$, because their vertex sets coincide: Lemma~\ref{lem:vertices-flow} asserts that the vertices of $\mathcal F_G(e_i - e_{n+1})$ are unit flows on paths $\mathsf p$ from $i$ to $n+1$; under $\varphi$, the vertex of $\mathcal F_G(e_i - e_{n+1})$ corresponding to $\mathsf p$ is mapped to the vertex of $\Delta_{S_{G,i}}$ corresponding to the (unique) edge in $\mathsf p$ that is incident to $n+1$. The claim $\varphi(\mathcal F_G(\mathbf a)) = \mathcal P(G;\mathbf a)$ follows.
\end{proof}

We note that a special case of Proposition~\ref{prop:flow-project-gp} was considered in~\cite[Section 4]{ms2017}.

\begin{definition}
\label{defn:ef}
For $i \in [n]$, let $I_i$ denote the set of $M(i,n+1)$ coordinates in $\mathbb R^{E(G)}$ corresponding to an edge connecting $i$ to $n+1$. For a flow  $x \in \mathcal F_G({\bf a})$, define the \textbf{escaping flow} vector $\ef x = (\ef x_1, \dots, \ef x_n) \in \mathbb R^n$ coordinatewise by
$$\ef x_i \overset{\rm def}= \sum_{j\in I_i}x_j.$$

For  $x\in\mathcal F_G(\mathbf a)$, and $\varphi$ as in Proposition \ref{prop:flow-project-gp}, define 
 $$\ef{\varphi(x)} \stackrel{\rm def}= \ef x.$$
\end{definition}

Note that if $i \not \in T_G$, or equivalently that if $I_i = \emptyset$, then $\ef x_i = 0$. Thus, we may regard $\ef x$ as a vector in $\RR^{T_G}$ (however, it will be useful to regard them as elements of $\RR^n$ whose coordinates indexed by $[n]\setminus T_G$ are zero).

Note also that $\ef x$ depends only on coordinates of $x\in\mathcal F_G(\mathbf a)\subseteq \RR^{E(G)}$ indexed by an edge $e \in S_G$. Hence $\ef{\varphi(x)} \stackrel{\rm def}= \ef x,$  is well defined since  $\varphi$ leaves the coordinates of $x$ corresponding to edges in $S_G$ unchanged.

\begin{proposition}
\label{prop:gp-project-smallgp}
There is a projection
\[
\psi\colon \mathcal F_G(\mathbf a) \twoheadrightarrow \mathcal Q(G;\mathbf a),
\]
where $\mathcal Q(G;\mathbf a)\subseteq \RR^{T_G}$ is the $y$-generalized permutahedron defined by
\[
\mathcal Q(G;\mathbf a)\stackrel{\rm def}= \sum_{i \in [n]}a_i\Delta_{T_{G,i}}.
\]
The map $\psi$ is given by sending $x\mapsto \ef x$. The map $\psi$ factors through $\varphi$, that is, the following diagram commutes:
\begin{center}\begin{tikzcd}
\mathcal F_G({\bf a}) \arrow[rr, "\varphi"] \arrow[rrrr, "\psi", bend right] &  & \mathcal P(G;{\bf a}) \arrow[rr, "\varphi(x)\mapsto\ef{\varphi(x)}"] &  & \mathcal Q(G;{\bf a})
\end{tikzcd}\end{center}
\end{proposition}

\begin{proof}
As in the proof of Proposition~\ref{prop:flow-project-gp}, it will suffice to show $\psi(\mathcal F_G(e_i - e_{n+1})) = \Delta_{T_{G,i}}$. Lemma~\ref{lem:vertices-flow} asserts that that the vertices of $\mathcal F_G(e_i - e_{n+1})$ are unit flows on paths $\mathsf p$ from $i$ to $n+1$; under $\psi$, the vertex of $\mathcal F_G(e_i - e_{n+1})$ corresponding to $\mathsf p$ is mapped to the vertex of $\Delta_{T_{G,i}}$ corresponding to the (unique) vertex $t$ of $G$ for which $\mathsf p$ contains an edge from $t$ to $n+1$.

That the diagram commutes is the statement that $\ef{\varphi(x)} \overset{\rm def}=\ef x$ is well defined, as discussed after Definition~\ref{defn:ef}.
\end{proof}

\section{The fibers of $\varphi$ and $\psi$}
\label{subsec:the-fibers-of-our-projection}

In order to study $\sigma_{G(\mathbf a)}^\varphi$ and $\sigma_{G(\mathbf a)}^\psi$ as defined in \eqref{def:1} and \eqref{def:2}, we  rewrite them as in equations \eqref{def:11} and \eqref{def:22} below; the validity of these equations follows from  Propositions \ref{prop:flow-project-gp} and \ref{prop:gp-project-smallgp}.    Equations \eqref{def:11} and \eqref{def:22} make it evident that in order to explicitly compute  the coefficients of the monomials appearing in $\sigma_{G(\mathbf a)}^\varphi$ and $\sigma_{G(\mathbf a)}^\psi$ (Corollary~\ref{cor:sigmaphi-sigmapsi-coeffs}) we need to compute the fibers of $\varphi$ and $\psi$, which is what we accomplish in  Theorem~\ref{thm:fibers-of-varphi} and Corollary~\ref{cor:fibers-of-psi}, respectively. 

 For brevity of notation, we index the coordinates of a point $\mathbf x \in \RR^{S_G}$ with $(i;k)$, which is shorthand for the edge $(i,n+1;k) \in S_G$. We define the polynomials
\begin{equation}\label{def:11}
\sigma_{G(\mathbf a)}^\varphi(x_{i;k})=\sum_{\mathbf p\in\mathcal P(G;\mathbf a)\cap\ZZ^{S_G}}(\varphi^{-1}(\mathbf p)\cap \ZZ^{E(G)})\mathbf x^{\mathbf p},
\end{equation}
and
\begin{equation}\label{def:22}
\sigma_{G(\mathbf a)}^\psi(x_i)=\sum_{\mathbf p\in\mathcal Q(G;\mathbf a)\cap\ZZ^{T_G}}(\psi^{-1}(\mathbf p)\cap\ZZ^{E(G)})\mathbf x^{\mathbf p}.
\end{equation}

\begin{theorem}
\label{thm:fibers-of-varphi}
Given a point $x \in \mathcal P(G;\mathbf a)$, the preimage $S_x\overset{\rm def}= \varphi^{-1}(x)\cap \mathcal F_G({\bf a})$ is a translation of the flow polytope $\mathcal F_G(a_1 - \ef x_1, \dots, a_n - \ef x_n,0)$. For $x \in \ZZ^{S_G}$, $S_x$ is integrally equivalent to $\mathcal F_G(a_1 - \ef x_1, \dots, a_n - \ef x_n,0)$.
\end{theorem}

We emphasize that $\mathbf a = (a_1, \dots, a_n, -\sum_{i=1}^n a_i)$ with $a_i \geq 0$, and that for any $x\in\mathcal P(G;\mathbf a)$ we have 
\[
\sum_{i=1}^n \ef x_i = \sum_{i=1}^n\sum_{j\in I_i} x_j = \sum_{j\in S_G}x_j = \sum_{i=1}^na_i,
\]
with the second equality by the fact that $\bigsqcup_i I_i = S_G$ and the last equality by the definition of $\mathcal P(G;\mathbf a)$.

\begin{proof}
Let $\varphi^\perp\colon\RR^{E(G)}\to\RR^{E(G)}$ denote the projection sending components corresponding to edges in $S_G$ to zero. Note that $\varphi$ and $\varphi^\perp$ project $\RR^{E(G)}$ to orthogonal complements, so $\varphi^\perp$ is necessarily an injection from $S_x$ onto its image (since points in $S_x$ are all mapped to $x$ by $\varphi$). To clean up notation, we will write $z_i = a_i - \ef x_i$.

Restricting an $\mathbf a$-flow in $S_x$ onto just the edges in $G|_{[n]}$ gives a (nonnegative) flow with netflow precisely $a_i - \ef x_i$ on vertex $i$. Hence, $\varphi^\perp$ is a map $S_x\hookrightarrow\mathcal F_G(z_1, \dots, z_n,0)$; furthermore, the inverse $\mathcal F_G(z_1, \dots, z_n,0) \to S_x$ is translation by
\[
\tilde x = (\tilde x_e)_{e\in E(G)}\in \RR^{E(G)}; \hspace{1cm} \tilde x_e \stackrel{\rm def}= \begin{cases} x_e&\text{ if } e \in S_G\\ 0 &\text{otherwise}\end{cases}
\]
Hence, $S_x$ is equal to $\mathcal F_G(z_1, \dots, z_n, 0)$ up to translation by $\tilde x$. Furthermore, if $x \in \ZZ^{S_G}$, then translation by $\tilde x \in \ZZ^{E(G)}$ is an integral equivalence $S_x\equiv\mathcal F_G(z_1, \dots, z_n,0)$.
\end{proof}

\begin{corollary}
\label{cor:fibers-of-psi}
Given a point $x \in \mathcal Q(G;{\bf a})$, the preimage $T_x\stackrel{\rm def}= \psi^{-1}(x)\cap \mathcal F_G({\bf a})$ is equal to $\mathcal F_G(a_1 - x_1, \dots, a_n - x_n,0)\times\prod_{i\in T_G}x_i\Delta_{I_i}$.
\end{corollary}
\begin{proof}
Observe that $\varphi^\perp(T_x) = \mathcal F_G(a_1 - x_1, \dots, a_n - x_n,0)$, since an $\mathbf a$-flow in $T_x$ restricted onto just the edges in $G|_{[n]}$ gives a flow with netflow precisely $a_i - x_i$ on vertex $i$. The fiber $\psi^{-1}(p)\cap \mathcal F_G(\mathbf a)$ of any point $p \in \mathcal F_G(a_1 - x_1, \dots, a_n - x_n, 0)$ is equal to $\{p\}\times\prod_{i\in T_G}x_i\Delta_{I_i}$. The claim follows.
\end{proof}

\begin{corollary}
\label{cor:sigmaphi-sigmapsi-coeffs}
We have
\[
\sigma_{G(\mathbf a)}^\varphi(x_{i;k}) =\sum_{\mathbf p\in\mathcal P(G;\mathbf a)\cap\ZZ^{S_G}}K_G(a_1 - \ef{\mathbf p}_1, \dots, a_n - \ef{\mathbf p}_n, 0)\mathbf x^{\mathbf p},
\]
and
\[
\sigma_{G(\mathbf a)}^\psi(x_i) =\sum_{\mathbf p\in\mathcal Q(G;\mathbf a)\cap\ZZ^{T_G}}K_G(a_1 -\ef{\mathbf p}_1, \dots, a_n - \ef{\mathbf p}_n)\binom{|I_1|+x_1 - 1}{x_1}\dots\binom{|I_n|+x_n-1}{x_n}\mathbf x^{\mathbf p}.
\]
\end{corollary}

\begin{proof}
The number of integer points of $\mathcal F_G(\mathbf a)$ is given by the Kostant partition function $K_G(\mathbf a)$. Combining this fact with Theorem~\ref{thm:fibers-of-varphi} and Corollary~\ref{cor:fibers-of-psi} gives the desired result.
\end{proof}

\section{Normalized projections of integer point transforms are Lorentzian}
\label{sec:?}
In this section, we show that $N(\sigma_{G(\mathbf a)}^\varphi)$ and $N(\sigma_{G(\mathbf a)}^\psi)$ are Lorentzian; see Theorem~\ref{thm:sigmaphi-sigmapsi-lorentzian}. In order to prove this we begin with a series of reductions (Proposition~\ref{prop:varphi-implies-psi} and Lemma~\ref{lem:G^--implies-G}). Then, a combinatorial symmetry (Lemma~\ref{lem:graph-flip}) allows us to realize Hessians   of repeated partial derivatives of $\sigma_{G(\mathbf a)}^\varphi$ as Hessians of repeated partial derivatives of volume polynomials.

 We begin by formally stating the main result of this section. 
 
 \begin{theorem}
\label{thm:sigmaphi-sigmapsi-lorentzian}
The polynomials $(N(\sigma_{G({\bf a})}^{\varphi}))(x_{i;k})$ and $(N(\sigma_{G({\bf a})}^{\psi}))(x_i)$ are Lorentzian. 
\end{theorem}

As $N(\sigma_{G(\mathbf a)}^\varphi)$ is Lorentzian, the coefficients of $\sigma_{G(\mathbf a)}^\varphi$ satisfy a log-concavity inequality (see Proposition~\ref{lem:log-concavity-of-coeffs}), which is equivalent to: 

\begin{corollary} [cf.\ {\cite[Proposition 11]{hmms2019}}]
\label{cor:kostant-log-concave}
For any directed graph $G$ on the vertex set  $[n]$  and for any $\mathbf v  \in \ZZ^n$ we have:
\[
K_G(\mathbf v)^2 \geq K_G(\mathbf v + e_i - e_j)K_G(\mathbf v - e_i + e_j)
\]
for every $i,j \in [n]$.
\end{corollary}

Note that Corollary \ref{cor:kostant-log-concave} also follows from the classical Alexandrov-Fenchel inequalities for mixed volumes, since $K_G(\mathbf v)$ can be seen as  mixed volumes of Minkowski sums of flow polytopes. 

A first stepping stone towards Theorem~\ref{thm:sigmaphi-sigmapsi-lorentzian} is to reduce to the problem of showing $N(\sigma_{G(\mathbf a)}^\varphi)$ is Lorentzian for all $G$; this is the content of Proposition~\ref{prop:varphi-implies-psi}. In order to do this we introduce the following construction.

\begin{definition}
\label{defn:G-ex}
For a graph $G$, we denote by $G^{\textup{ex}} = (V^{\textup{ex}}, E^{\textup{ex}})$ the graph obtained from $G$ by adding formal vertices $i^{\textup{ex}}$ for each vertex $i\in T_G$, by replacing edges $(i,n+1;j) \in S_G$ with edges $(i,i^{\textup{ex}};j)$, and by adding edges $(i^{\textup{ex}},n+1;1)$ for each $i^{\textup{ex}} \in T_G$. Formally, we have
\begin{align*}
V^{\textup{ex}}&\overset{\rm def}=[n]\sqcup\{i^{\textup{ex}}\colon i \in T_G\}\sqcup\{n+1\},\\
E^{\textup{ex}}&\overset{\rm def}=(E\setminus S_G) \sqcup\{(i,i^{\textup{ex}};j)\colon (i,n+1;j) \in S_G\}\sqcup\{(i^{\textup{ex}},n+1;1)\colon i \in T_G\}.
\end{align*}
See Figure~\ref{fig:G-ex} for an example. The graph $G$ can be recovered from $G^{\textup{ex}}$ by a series of contractions, so we call $G^{\textup{ex}}$ the \textbf{extension} of $G$.
\end{definition}

\begin{figure}[ht]
\begin{center}
\includegraphics[scale=0.55]{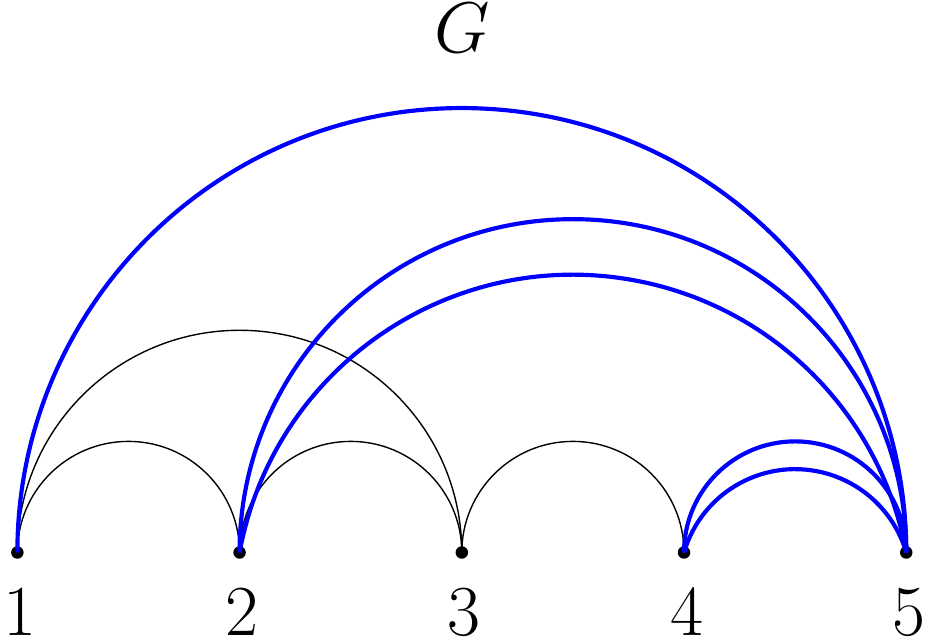}\hspace{1cm} \includegraphics[scale=0.55]{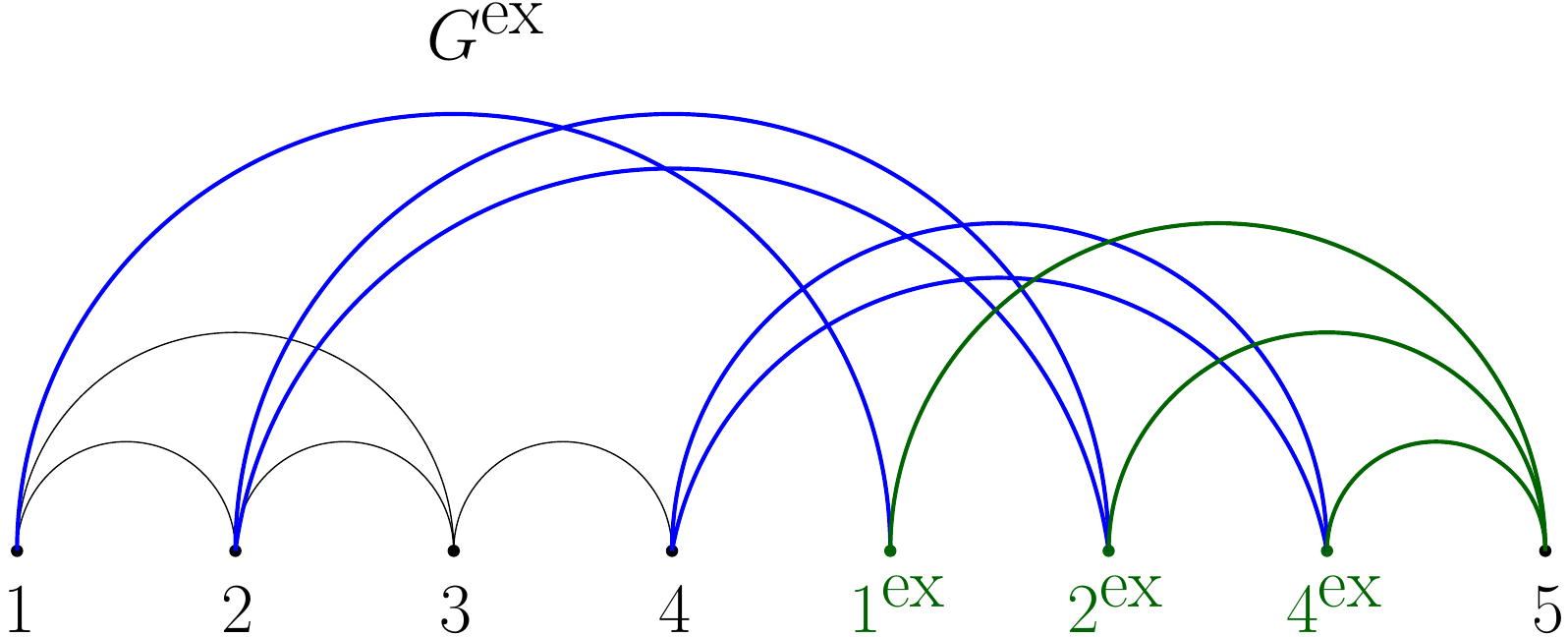}
\end{center}
\caption{The graph $G$ in Figure~\ref{fig:STst}, along with $G^{\textup{ex}}$, defined in Definition~\ref{defn:G-ex}.}
\label{fig:G-ex}
\end{figure}

\begin{definition}
For any two vectors $\mathbf p = (p_1, \dots, p_m) \in \RR^m$ and $\mathbf q = (q_1, \dots, q_n)\in \RR^n$, we denote by $\mathbf p \oplus\mathbf q \in \RR^{m+n}$ their \textbf{concatenation}, that is,
\[
\mathbf p \oplus\mathbf q = (p_1, \dots, p_m, q_1, \dots, q_n) \in \RR^{m+n}.
\]
For a netflow $\mathbf a$ for $G$ satisfying the conventions of this paper, we denote by $\mathbf a^{\textup{ex}}$ the netflow for $G^{\textup{ex}}$ given by
\[
\mathbf a|_{[n]} \oplus \mathbf 0_{T_{G^{\textup{ex}}}} \oplus -a_{n+1} = (a_1, \dots, a_n, \underbrace{0,\dots,0}_{|T_{G^{\textup{ex}}}|\textup{ many}},-a_{n+1});
\]
note that $\mathbf a^{\textup{ex}}$ also satisfies the conventions of this paper.
\end{definition}

\begin{lemma}
\label{lem:supp-iso}
The bijection $T_G\leftrightarrow S_{G^{\textup{ex}}}$ given by $i\leftrightarrow (i^{\textup{ex}},n+1;1)$ induces an isomorphism on the real vector spaces $\RR^{T_G}$ and $\RR^{S_{G^{\textup{ex}}}}$ by renaming basis elements according to the bijection. This isomorphism restricts to an integral equivalence $\mathcal Q(G;\mathbf a) \equiv \mathcal P(G^{\textup{ex}};\mathbf a^{\textup{ex}})$.
\end{lemma}
 
\begin{proof}
By definition,
\[
\mathcal Q(G;\mathbf a) = \sum_{i\in [n]}a_i\Delta_{T_{G,i}} \hspace{0.4cm}\textup{ and }\hspace{0.4cm} \mathcal P(G^{\textup{ex}};\mathbf a^{\textup{ex}}) = \sum_{i\in[n]\sqcup T_{G^{\textup{ex}}}} a^{\textup{ex}}_i \Delta_{S_{G^{\textup{ex}},i}}.
\]
Since for every $i \in T_{G^{\textup{ex}}}$ we have $a^{\textup{ex}}_i = 0$, and for every $i \in [n]$ we have $a^{\textup{ex}}_i = a_i$, we may write
\[
\mathcal P(G^{\textup{ex}};\mathbf a^{\textup{ex}}) = \sum_{i\in[n]}a_i\Delta_{S_{G^{\textup{ex}}},i};
\]
furthermore, we have $S_{G^{\textup{ex}},i} = \{(i^{\textup{ex}},n+1;1)\colon i \in T_{G,i}\}$. Thus, the isomorphism sends $\Delta_{T_{G,i}}$ to $\Delta_{S_{G^{\textup{ex}},i}}$; passing to the Minkowski sum, we obtain the integral equivalence $\mathcal Q(G;\mathbf a) \equiv \mathcal P(G^{\textup{ex}};\mathbf a^{\textup{ex}})$.
\end{proof}

\begin{lemma}
\label{lem:fiber-iso}
The bijection $E(G)\leftrightarrow E(G^{\textup{ex}})\setminus S_{G^{\textup{ex}}}$ given by sending an edge $(i,n+1;k)\in S_G$ to $(i,i^{\textup{ex}};k) \in E(G^{\textup{ex}})\setminus S_{G^{\textup{ex}}}$ and an edge $(i,j;k) \in E(G)\setminus S_G$ to $(i,j;k) \in E(G^{\textup{ex}})\setminus S_{G^{\textup{ex}}}$ induces an isomorphism on the real vector spaces spanned by $E(G)$ and $E(G^{\textup{ex}})\setminus S_{G^{\textup{ex}}}$ by renaming basis elements according to the bijection. For every $\mathbf q \in \ZZ^{T_G}$, this isomorphism restricts to an integral equivalence
\begin{equation}
\label{eq:lin-equiv-fibers}
\mathcal F_{G|_{[n]}}(\mathbf a|_{[n]} - \mathbf q)\times\prod_{i\in T_G}q_i\Delta_{I_i} \equiv\mathcal F_{G^{\textup{ex}}|_{[n]\sqcup T_{G^{\textup{ex}}}}}(\mathbf a|_{[n]} \oplus(-\mathbf q)).
\end{equation}
\end{lemma}

In light of Corollary~\ref{cor:fibers-of-psi}, note that the left side of Equation~\eqref{eq:lin-equiv-fibers} is the fiber of $\mathbf q$ under $\mathcal F_G(\mathbf a)\to\mathcal Q(G;\mathbf a)$. For brevity of notation, let us temporarily denote by $\widetilde{\mathbf q}\in \ZZ^{S_{G^{\textup{ex}}}}$ the image of $\mathbf q\in \ZZ^{T_G}$ under the isomorphism in Lemma~\ref{lem:supp-iso}. In this notation, Theorem~\ref{thm:fibers-of-varphi} implies that the right side of Equation~\eqref{eq:lin-equiv-fibers} is (integrally equivalent to) the fiber of $\widetilde{\mathbf q}\in \ZZ^{S_{G^{\textup{ex}}}}$ under $\mathcal F_G^{\textup{ex}}(\mathbf a^{\textup{ex}}) \to \mathcal P(G^{\textup{ex}},\mathbf a^{\textup{ex}})$.

\begin{proof}[Proof of Lemma~\ref{lem:fiber-iso}]
A point $f \in \mathcal F_{G|_{[n]}}(\mathbf a|_{[n]} - \mathbf q)\times\prod_{i\in T_G}q_i\Delta_{I_i}$ can be interpreted as a flow in $\mathcal F_G(\mathbf a)$ with outflow $q_i$ at each vertex $i \in T_G$, by Corollary~\ref{cor:fibers-of-psi}. Under the isomorphism in Lemma~\ref{lem:fiber-iso}, $f$ gets mapped to a flow in $G^{\textup{ex}}|_{[n]\sqcup T_{G^{\textup{ex}}}}$ with netflow $a_i$ at each vertex $i \in [n]$ and netflow $-q_i$ at each vertex $i^{\textup{ex}} \in T_{G^{\textup{ex}}}$. In other words, the image is in $\mathcal F_{G^{\textup{ex}}|_{[n]\sqcup T_{G^{\textup{ex}}}}}(\mathbf a|_{[n]} \oplus(-\mathbf q))$.

Conversely, the preimage of a flow $\mathcal F_{G^{\textup{ex}}|_{[n]\sqcup T_{G^{\textup{ex}}}}}(\mathbf a|_{[n]} \oplus(-\mathbf q))$ is a flow in $\mathcal F_G(\mathbf a)$ with netflow $q_i$ at each vertex $i \in T_G$; hence by Corollary~\ref{cor:fibers-of-psi} the preimage of $f$ is in $\mathcal F_{G|_{[n]}}(\mathbf a|_{[n]} - \mathbf q)\times\prod_{i\in T_G}q_i\Delta_{I_i}$.
\end{proof}

\begin{lemma}
\label{lem:sigmapsi-sigma-gex-phi}
The bijection $T_G\leftrightarrow S_{G^{\textup{ex}}}$ given by $i\leftrightarrow (i^{\textup{ex}},n+1;1)$ induces an isomorphism on the polynomial rings $\RR[(x_i)_{i\in T_G}]$ and $\RR[(x_i)_{i\in S_{G^{\textup{ex}}}}]$ by renaming variables according to the bijection. Under this isomorphism, the polynomial $N(\sigma_{G(\mathbf a)}^\psi)$ is sent to $N(\sigma_{G^{\textup{ex}}(\mathbf a^{\textup{ex}})}^\varphi)$.
\end{lemma}

\begin{proof}
Explicitly, we need to show that the polynomials
\begin{equation}
\label{eq:sigmapsi-explicit}
N(\sigma_{G(\mathbf a)}^\psi) = \sum_{\mathbf q \in \mathcal Q(G;\mathbf a)\cap \ZZ^{T_G}}(\psi^{-1}(\mathbf q)\cap \ZZ^{E(G)})\mathbf x^{\mathbf q}
\end{equation}
and
\begin{equation}
\label{eq:sigmaphi-Gex-explicit}
N(\sigma_{G^{\textup{ex}}(\mathbf a^{\textup{ex}})}^\varphi) = \sum_{\mathbf p \in\mathcal P(G^{\textup{ex}},\mathbf a^{\textup{ex}})\cap \ZZ^{S_{G^{\textup{ex}}}}}(\varphi^{-1}(\mathbf p)\cap \ZZ^{E(G^{\textup{ex}})})\mathbf x^{\mathbf p}
\end{equation}
agree after renaming variables according to the bijection. We stress that the map $\psi$ in Equation~\eqref{eq:sigmapsi-explicit} is the projection $\mathcal F_G(\mathbf a)\to \mathcal Q(G;\mathbf a)$, whereas the map $\varphi$ in Equation~\eqref{eq:sigmaphi-Gex-explicit} is the projection $\mathcal F_{G^{\textup{ex}}}(\mathbf a^{\textup{ex}}) \to \mathcal P(G^{\textup{ex}};\mathbf a^{\textup{ex}})$.

By Lemma~\ref{lem:supp-iso}, the monomials $\mathbf x^{\mathbf q}$ appearing in Equation~\eqref{eq:sigmapsi-explicit} and the monomials $\mathbf x^{\mathbf p}$ appearing in Equation~\eqref{eq:sigmaphi-Gex-explicit} correspond to each other under the isomorphism $\RR[(x_i)_{i\in T_G}] \cong \RR[(x_i)_{i\in S_{G^{\textup{ex}}}}]$.

By Lemma~\ref{lem:fiber-iso}, the fibers $\psi^{-1}(\mathbf q)$ appearing in Equation~\eqref{eq:sigmapsi-explicit} and the corresponding fibers $\varphi^{-1}(\mathbf p)$ appearing in Equation~\eqref{eq:sigmaphi-Gex-explicit} are integrally equivalent. Hence the coefficients of the monomials appearing in Equations~\eqref{eq:sigmapsi-explicit} and~\eqref{eq:sigmaphi-Gex-explicit} match.
\end{proof}

\begin{proposition}
\label{prop:varphi-implies-psi}
Suppose $N(\sigma_{G(\mathbf a)}^\varphi)$ is Lorentzian for every $G$. Then $N(\sigma_{G(\mathbf a)}^\psi)$ is Lorentzian for every $G$.
\end{proposition}

\begin{proof}
Lemma~\ref{lem:sigmapsi-sigma-gex-phi} asserts that up to renaming variables, we have the equality
\[
N(\sigma_{G(\mathbf a)}^\psi) = N(\sigma_{G^{\textup{ex}}(\mathbf a^{\textup{ex}})}^\varphi).
\]
By assumption, $N(\sigma_{G^{\textup{ex}}(\mathbf a^{\textup{ex}})}^\varphi)$ is Lorentzian.
\end{proof}

\begin{lemma}
\label{lem:simple-at-sink-to-matrix}
Let $\mathbf d \in (\sum_{i=1}^n a_i - 2)\Delta_{S_G}\cap\ZZ^{S_G}$ be an integer point in the scaled coordinate simplex of $S_G$. Suppose that the $|S_G|\times |S_G|$ matrix
\[
K_{\mathbf d} \stackrel{\rm def}=(k_{i,j})_{i,j \in S_G}; \hspace{1cm} k_{(i_1;k_1),(i_2;k_2)} \stackrel{\rm def}= K_{G|_{[n]}}(\mathbf a|_{[n]} - \ef{\mathbf d} - e_{i_1} - e_{i_2})
\]
has at most one positive eigenvalue. Then $N(\sigma_{G(\mathbf a)}^\varphi)$ is Lorentzian. 
\end{lemma}

\begin{proof}
The support of $N(\sigma_{G(\mathbf a)}^\varphi)$ is M-convex by Proposition~\ref{prop:flow-project-gp}. 

By Corollary~\ref{cor:sigmaphi-sigmapsi-coeffs}, the $ij$-th element of $K_{\mathbf d}$ is the coefficient of $\mathbf x^{\mathbf d + e_i +e_j}$ in $\sigma_{G(\mathbf a)}^\varphi$; equivalently, $K_{\mathbf d}$ is the Hessian of $\partial_{\mathbf d}N(\sigma_{G(\mathbf a)}^\varphi)$. Since, by assumption, $K_{\mathbf d}$ has at most one positive eigenvalue, Lemma~\ref{lem:lorentzian-criterion} asserts that $N(\sigma_{G(\mathbf a)}^\varphi)$ is Lorentzian.
\end{proof}

\begin{definition}
\label{defn:G^-}
For a graph $G$ as in the conventions of this section, denote by $G^-$ the graph with vertex set $[n+1]$ and edge set
\[
E(G^-) = E(G|_{[n]}) \sqcup \{e \in S_G\colon e = (i,n+1;1)\}.
\]
In other words, $G^-$ is obtained from $G$ by replacing, for each $i \in [n]$ with $M(i,n+1) \geq 1$, the set of edges connecting $i$ to the sink with a single edge connecting $i$ to the sink. See Figure~\ref{fig:example-G^-} for an example. Note that since $G^-$ has at most one edge connecting $i$ to $n+1$ for any $i$, we have $S_{G^-} = T_{G^-} = T_G$; we index the variables appearing in $\sigma_{G^-(\mathbf a)}^\varphi$ with $(x_i)_{i \in T_G}$.
\end{definition}

\begin{figure}[ht]
\begin{center}
\includegraphics[scale=0.6]{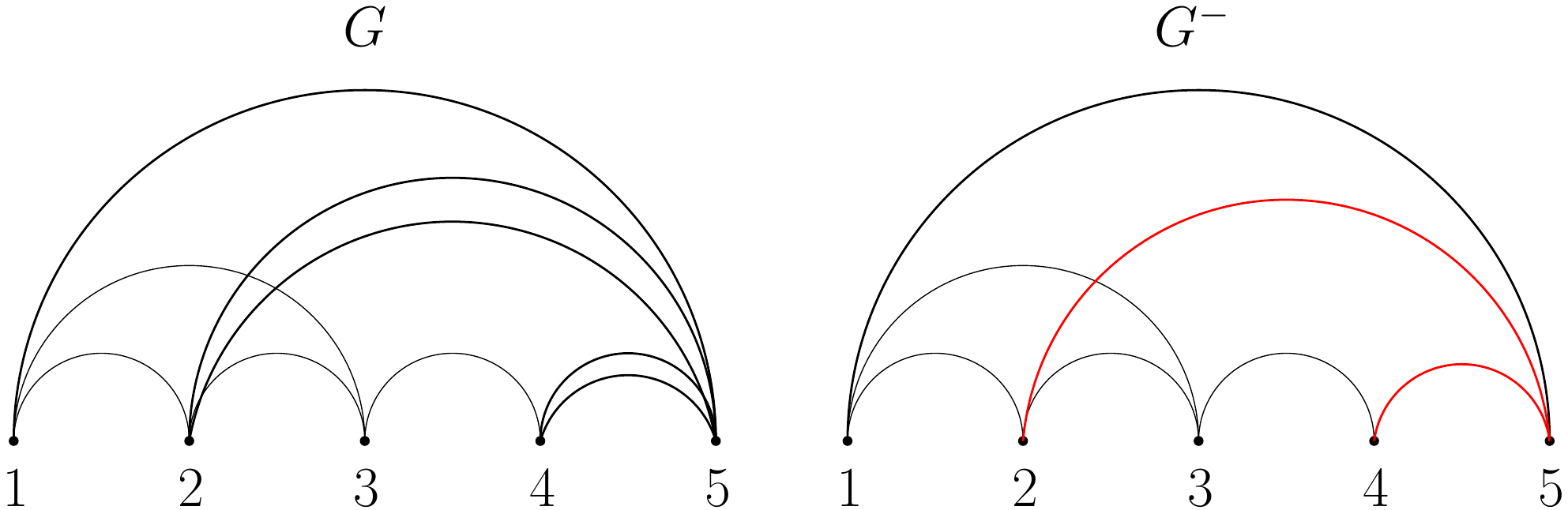}
\caption{The graph $G$ from Figures~\ref{fig:STst} and~\ref{fig:G-ex}. The graph $G^-$ constructed from $G$ is shown beside it; see Definition~\ref{defn:G^-}.}
\label{fig:example-G^-}
\end{center}
\end{figure}

\begin{lemma}
\label{lem:G^--implies-G}
Suppose $(N(\sigma_{G^-(\mathbf a)}^\varphi))(x_i)$ is Lorentzian. Then $(N(\sigma_{G(\mathbf a)}^\varphi))(x_{i;k})$ is also Lorentzian.
\end{lemma}

\begin{proof}
The support of $N(\sigma_{G(\mathbf a)}^\varphi)$ is M-convex by Proposition~\ref{prop:flow-project-gp}.

By Lemma~\ref{lem:simple-at-sink-to-matrix}, we need to show that for every $\mathbf d \in (\sum_{i=1}^n a_i - 2)\Delta_{S_G}\cap\ZZ^{S_G}$, the $|S_G|\times |S_G|$ matrix
\[
K_{\mathbf d} =(k_{i,j})_{i,j \in S_G}; \hspace{1cm} k_{(i_1;k_1),(i_2;k_2)} = K_{G|_{[n]}}(\mathbf a|_{[n]} - \ef{\mathbf d} - e_{i_1} - e_{i_2})
\]
has at most one positive eigenvalue. The matrix $K_{\mathbf d}$  is obtained from the $|T_G|\times|T_G|$ matrix
\[
K_{\ef{\mathbf d}}^- \stackrel{\rm def}=(k_{i,j}^-)_{i,j \in T_G}; \hspace{1cm} k_{i,j}^- \stackrel{\rm def}= K_{G|_{[n]}}(\mathbf a|_{[n]} - \ef{\mathbf d} - e_i - e_j)
\]
first by repeating the $i$th row $M(i,n+1)$ many times for each $i$, and then by repeating the $i$th column $M(i,n+1)$ many times for each $i$. Note that the rank of $K_{\ef{\mathbf d}}^-$ is equal to the rank of $K_{\mathbf d}$; we write
\[
r\overset{\rm def}=\textup{rank}(K_{\ef{\mathbf d}}^-) = \textup{rank}(K_{\mathbf d}).
\]

Observe, by Corollary~\ref{cor:sigmaphi-sigmapsi-coeffs} that the $ij$-th entry of $K_{\ef{\mathbf d}}^-$ is the coefficient of $\mathbf x^{\ef{\mathbf d} + e_i + e_j}$ in $\sigma_{G^-(\mathbf a)}^\varphi$. By assumption, $N(\sigma_{G^-(\mathbf a)}^\varphi)$ is Lorentzian; hence, Lemma~\ref{lem:lorentzian-criterion} asserts that $K_{\ef{\mathbf d}}^-$ has at most one positive eigenvalue for any $\ef{\mathbf d} \in (\sum_{i=1}^na_i - 2)\Delta_{T_G}\cap\ZZ^{T_G}$. In particular it has at least $r-1$ negative  eigenvalues. Note also that $K_{\ef{\mathbf d}}^-$ is a principal submatrix of $K_{\mathbf d}$; by Cauchy's Interlacing Theorem (Proposition~\ref{prop:cauchy-interlacing}), the eigenvalues $\alpha_1 \leq \alpha_2 \leq \dots \leq \alpha_{|S_G|}$ of $K_{\mathbf d}$ and the eigenvalues $\beta_1 \leq \dots \leq \beta_{|T_G|}$ of $K_{\ef{\mathbf d}}^-$ satisfy
\[
\alpha_i \leq \beta_i \hspace{1cm}\textup{ for all } 1 \leq i \leq |T_G|.
\]
Since $K_{\ef{\mathbf d}}^-$ has at least $r-1$ negative  eigenvalues,
\[
\alpha_i \leq \beta_i < 0 \hspace{1cm}\textup{  for all } 1 \leq i \leq r-1,
\] 
so $K_{\mathbf d}$ also has at least $r-1$ negative  eigenvalues. Furthermore, $K_{\mathbf d}$ has rank $r$. Hence, $K_{\mathbf d}$ also has at most one positive eigenvalue, and $(N(\sigma_{G(\mathbf a)}^\varphi))(x_{i;k})$ is Lorentzian.
\end{proof}

\begin{example}
\label{ex:K-d,K-efd}
Let $G$ be as in Figure~\ref{fig:example-G^-} and $\mathbf a = (2,1,1,1,-5)$. Let $\mathbf d \in 3\Delta_{S_G}\cap\ZZ^{S_G}$ be the vector $e_{2;1} + e_{2;2} + e_{4;2}$; this integer vector takes the value 1 on the edges $(2,5;1), (2,5;2), (4,5;2) \in S_G$ and takes the value 0 everywhere else. Thus $\mathbf a|_{[n]} - \ef{\mathbf d} = (2,-1,1,0)$. The matrix $K_{\mathbf d}$ is given by
\[
\begin{bmatrix}
k_{(1;1),(1;1)} & k_{(1;1),(2;1)} & k_{(1;1),(2;2)} & k_{(1;1),(4;1)} & k_{(1;1),(4;2)} \\ 
k_{(2;1),(1;1)} & k_{(2;1),(2;1)} & k_{(2;1),(2;2)} & k_{(2;1),(4;1)} & k_{(2;1),(4;2)} \\ 
k_{(2;2),(1;1)} & k_{(2;2),(2;1)} & k_{(2;2),(2;2)} & k_{(2;2),(4;1)} & k_{(2;2),(4;2)} \\ 
k_{(4;1),(1;1)} & k_{(4;1),(2;1)} & k_{(4;1),(2;2)} & k_{(4;1),(4;1)} & k_{(4;1),(4;2)} \\ 
k_{(4;2),(1;1)} & k_{(4;2),(2;1)} & k_{(4;2),(2;2)} & k_{(4;2),(4;1)} & k_{(4;2),(4;2)} \\ 
\end{bmatrix}
=
\begin{bmatrix}
0 & 0 & 0 & 1 &1\\
0 & 0 & 0 & 1 & 1\\
0 & 0 & 0 & 1 & 1\\
1 & 1 & 1 & 2 & 2\\
1 & 1 & 1 & 2 & 2
\end{bmatrix}
\]
It is obtained from the matrix $K_{\ef{\mathbf d}}^-$ given by
\[
\begin{bmatrix}
k_{1,1}^- & k_{1,2}^- & k_{1,4}^- \\
k_{2,1}^- & k_{2,2}^- & k_{2,4}^- \\
k_{4,1}^- & k_{4,2}^- & k_{4,4}^-
\end{bmatrix}
=
\begin{bmatrix}
0 & 0 & 1 \\
0 & 0 & 1 \\
1 & 1 & 2
\end{bmatrix}
\]
first by repeating the second row $M(2,5) = 2$ times and repeating the third row $M(4,5) = 2$ times, to obtain
\[
\begin{bmatrix}
0 & 0 & 1\\
0 & 0 & 1\\
0 & 0 & 1\\
1 & 1 & 2\\
1 & 1 & 2
\end{bmatrix}
\]
and then repeating the second column $M(2,5) = 2$ times and repeating the third column $M(4,5) = 2$ times, to obtain
\[
\begin{bmatrix}
0 & 0 & 0 & 1 & 1\\
0 & 0 & 0 & 1 & 1\\
0 & 0 & 0 & 1 & 1\\
1 & 1 & 1 & 2 & 2\\
1 & 1 & 1 & 2 & 2
\end{bmatrix}
\]
The spectrum for $K_{\ef{\mathbf d}}^-$ is $\{-\sqrt 3 + 1, 0, \sqrt 3 + 1\}$ (which has at most one positive eigenvalue). 

In this example, the ranks of $K_{\ef{\mathbf d}}^-$ and $K_{\mathbf d}$ are both equal to $2$, thus they both have a total of $2$ nonzero eigenvalues.  The matrix $K_{\ef{\mathbf d}}^-$ is the principal submatrix of $K_{\mathbf d}$ corresponding to the 1st, 2nd, and 4th rows and columns of $K_{\mathbf d}$. Cauchy's Interlacing Theorem says that the smallest eigenvalue of $K_{\mathbf d}$ is at most $-\sqrt 3 + 1 < 0$. Hence $K_{\mathbf d}$ has at most one positive eigenvalue. \end{example}
 
\begin{definition}
For a graph $G$ as in the conventions of this section, denote by $G^r$ the graph obtained by ``flipping'' $G|_{[n]}$, that is, $V(G^r) = [n]$ and
\[
(i,j) \in E(G^r) \iff (n+1-j, n+1-i) \in E(G|_{[n]}).
\]
Equivalently, $G^r$ is obtained by relabeling the vertices of $G|_{[n]}$ by the map $i\mapsto n+1-i$ and reversing the orientation of edges. See Figure~\ref{fig:example-G^*} for an example.
\end{definition}

\begin{figure}[ht]
\begin{center}
\includegraphics[scale=0.6]{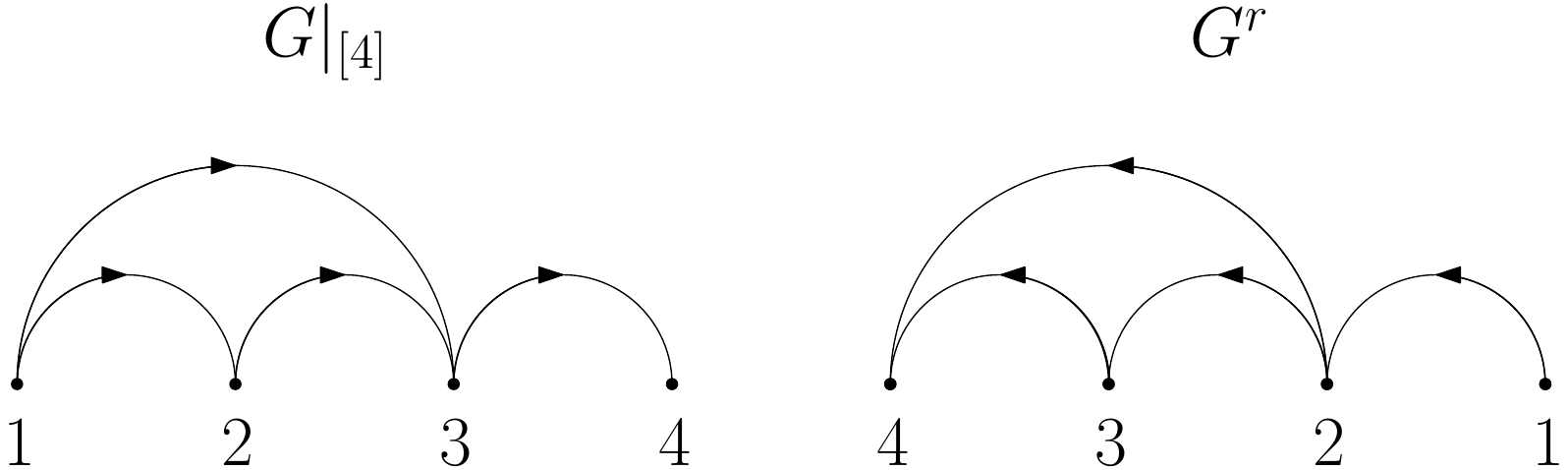}
\end{center}
\caption{The graphs $G|_{[4]}$ and $G^r$ are shown, for $G$ as in Figure~\ref{fig:example-G^-} and Example~\ref{ex:K-d,K-efd}.}
\label{fig:example-G^*}
\end{figure}

The symmetry between $G|_{[n]}$ and $G^r$ underpins the following lemma, crucial for the proof of Theorem~\ref{thm:sigmaphi-sigmapsi-lorentzian}:

\begin{lemma}[{\cite[Corollary 2.4]{mm2017}}]
\label{lem:graph-flip}
For every $c_1, \dots, c_n \in \ZZ_{\geq 0}$, the formula
\[
K_{G|_{[n]}}(c_1, c_2, \dots, c_{n-1}, c_n) = K_{G^r}(-c_n, -c_{n-1}, \dots, -c_2, -c_1)
\]
holds.
\end{lemma}

\begin{definition}
\label{defn:tilde-k}
Let $G^-$ be a graph satisfying the conventions of this section as well as $M(i,n+1) \leq 1$ for every $i$, see Definition~\ref{defn:G^-}. Denote by $P_T$ the permutation matrix corresponding to the order-reversing permutation $i\mapsto |T_G|+1-i$; this is the matrix consisting of 1's on the antidiagonal and 0 everywhere else. Observe that
\[
\widetilde K_{\ef{\mathbf d}}^- \stackrel{\rm def}= P_T K_{\ef{\mathbf d}}^-P_T = P_T^{-1}K_{\mathbf d} P_T
\]
has the same spectrum as $K_{\ef{\mathbf d}}^-$ since it is obtained by conjugation. Motivated by the following Proposition~\ref{prop:entrywise-formula}, as well as the fact that $\widetilde K_{\ef{\mathbf d}}^-$ is obtained from $K_{\ef{\mathbf d}}^-$ by permuting the rows and columns according to order-reversing permutation $i\mapsto |T_G| + 1 - i$, we choose to index the rows and columns of $\widetilde K_{\ef{\mathbf d}}^-$ by $\{n+1-i\colon i \in T_G\}$. See Example~\ref{ex:tilde-K-efd}.
\end{definition}

\begin{example}
\label{ex:tilde-K-efd}
Let $G$ be as in Figure~\ref{fig:example-G^-}, Example~\ref{ex:K-d,K-efd}, and Figure~\ref{fig:example-G^*}. The entries of $K_{\ef{\mathbf d}^-}$ are given by
\[
\begin{bmatrix}
k_{1,1}^- & k_{1,2}^- & k_{1,4}^- \\
k_{2,1}^- & k_{2,2}^- & k_{2,4}^- \\
k_{4,1}^- & k_{4,2}^- & k_{4,4}^-
\end{bmatrix}
=
\begin{bmatrix}
0 & 0 & 1 \\
0 & 0 & 1 \\
1 & 1 & 2
\end{bmatrix}.
\]
Then the matrix $\widetilde K_{\ef{\mathbf d}}^- = P_TK_{\ef{\mathbf d}}^- P_T$ is given by
\[
\begin{bmatrix}
\widetilde k_{4,4}^- & \widetilde k_{4,3}^- & \widetilde k_{4,1}^- \\
\widetilde k_{3,4}^- & \widetilde k_{3,3}^- & \widetilde k_{3,1}^- \\
\widetilde k_{1,4}^- & \widetilde k_{1,3}^- & \widetilde k_{1,1}^-
\end{bmatrix}
=
\begin{bmatrix} 0 & 0 & 1 \\ 0 & 1 & 0 \\ 1 & 0 & 0\end{bmatrix}
\begin{bmatrix}
0 & 0 & 1 \\
0 & 0 & 1 \\
1 & 1 & 2
\end{bmatrix}
\begin{bmatrix} 0 & 0 & 1 \\ 0 & 1 & 0 \\ 1 & 0 & 0\end{bmatrix}
=
\begin{bmatrix}
2 & 1 & 1 \\
1 & 0 & 0 \\
1 & 0 & 0
\end{bmatrix}.
\]
\end{example}

\begin{proposition}
\label{prop:entrywise-formula}
The entries of $\widetilde K_{\ef{\mathbf d}}^-$ are given by the following formula. Let
\[
\mathbf z=(\ef{\mathbf d}_n - a_n, \ef{\mathbf d}_{n-1} - a_{n-1}, \dots, \ef{\mathbf d}_2 - a_2, \ef{\mathbf d}_1 - a_1).
\]
Then the $i,j$-th entry $\widetilde k_{i,j}^-$ of $\widetilde K_{\ef{\mathbf d}}^-$ is $K_{G^r}(\mathbf z + e_i + e_j)$.
\end{proposition}

\begin{proof}
Note that the $i,j$-th entry $\widetilde k_{i,j}^-$ of $\widetilde K_{\ef{\mathbf d}}^-$ is precisely
\[
k_{n+1-i,n+1-j}^- = K_{G|_{[n]}}(\mathbf a|_{[n]} - \ef{\mathbf d} - e_{n+1-i} - e_{n+1-j}) = K_{G^r}(\mathbf z + e_i + e_j),
\]
where the last equality is an application of Lemma~\ref{lem:graph-flip}.
\end{proof}

The final piece required to prove Theorem~\ref{thm:sigmaphi-sigmapsi-lorentzian} is the existence of a quadratic Lorentzian polynomial whose Hessian is $\widetilde K_{\ef{\mathbf d}}^-$. We are ready to accomplish this now:

\begin{proof}[Proof of Theorem~\ref{thm:sigmaphi-sigmapsi-lorentzian}]
By Proposition~\ref{prop:varphi-implies-psi}, it suffices to show that $(N(\sigma_{G(\mathbf a)}^\varphi))(x_{i;k})$ is Lorentzian, and by Lemma~\ref{lem:G^--implies-G}, it suffices to show that $(N(\sigma_{G^-(\mathbf a)}^\varphi))(x_i)$ is Lorentzian. By Lemma~\ref{lem:simple-at-sink-to-matrix} applied to $G^-$, we need to show that $K_{\mathbf d} = K_{\ef{\mathbf d}}^-$ has at most one positive eigenvalue for every $\mathbf d \in (\sum_{i=1}^n a_i - 2)\Delta_{S_G}\cap\ZZ^{S_G} = (\sum_{i=1}^n a_i - 2)\Delta_{T_G}\cap\ZZ^{T_G}$. In light of the discussion in Definition~\ref{defn:tilde-k}, it suffices to show, for every lattice point $\ef{\mathbf d} \in (\sum a_i - 2)\Delta_{T_G}\cap\ZZ^{T_G}$, that the matrix $\widetilde K_{\ef{\mathbf d}}^-$ has at most one positive eigenvalue.

For brevity of notation, we introduce
\[
\mathbf z = (z_1, \dots, z_n) = (\ef d_n - a_n, \dots, \ef d_1 - a_1); \hspace{1cm} z_{\min}\stackrel{\rm def}= \min_{i \in [n]}z_i.
\]
Note that  $z_{\min} < 0$, since  $\sum z_i = -2$. Let $\widetilde G$ be the graph on the vertex set $[n+1-z_{\min}]$ with edges
\[
E(\widetilde G) = E(G^r)\cup\{(i,j)\colon i\leq j \text{ and } n+1 \leq j\}.
\]
Set
\begin{alignat*}{2}
&N\stackrel{\rm def}= n-z_{\min}; \qquad &&\widetilde{\mathbf z} \stackrel{\rm def}=(z_1, \dots, z_n, \underbrace{0,\dots,0}_{-z_{\min}});\\ &\widetilde{\mathbf x}\stackrel{\rm def}=(x_1, \dots, x_{N+1}); \qquad &&\widetilde{\mathbf o} \stackrel{\rm def}= (\text{outd}_1 - 1, \dots, \text{outd}_N - 1),
\end{alignat*}
where $\text{outd}_i$ denotes the outdegree of $\widetilde G$ at vertex $i$. Note that $\widetilde{\mathbf z} +\widetilde{\mathbf o} \geq \mathbf 0$, since for $i \leq n$ we have 
\[
\widetilde{\mathbf o}_i = \text{outd}_i - 1 \geq |\{(i,j)\colon n+1\leq j\leq N+1\}|-1 = N-n = -z_{\min},
\]
and for $i \geq n+1$ we have $\widetilde{\mathbf z}_i = 0$. The Baldoni-Vergne formulas, Theorem~\ref{thm:baldoni-vergne}, applied to $\widetilde G$ says that
\[
\vol\,\mathcal F_{\widetilde G}(\widetilde{\mathbf x}) = \sum_{\substack{\mathbf j\colon \mathbf j\geq-\widetilde{\mathbf o}\\j_1 + \dots + j_N = 0}}K_{\widetilde G}(\mathbf j)\frac{(\widetilde{\mathbf x}|_{[N]})^{\mathbf j + \widetilde{\mathbf o}}}{(\mathbf j + \widetilde{\mathbf o})!}.
\]
By Theorem~\ref{lem:vol-is-lorentzian}, $\vol\,\mathcal F_{\widetilde G}(\widetilde{\mathbf x})$ is Lorentzian. Hence, so is
\[
\partial_{\widetilde{\mathbf z} + \widetilde{\mathbf o}} \vol\,\mathcal F_{\widetilde G}(\widetilde{\mathbf x}) = \sum_{\substack{\mathbf j\colon \mathbf j \geq \widetilde{\mathbf z}\\j_1 + \dots + j_N = 0}}K_{\widetilde G}(\mathbf j)\frac{(\widetilde{\mathbf x}|_{[N]})^{\mathbf j - \widetilde{\mathbf z}}}{(\mathbf j - \widetilde{\mathbf z})!},
\]
where the equality is an application of Lemma~\ref{lem:NDN-is-coeff}. Let $A$ be the $N\times N$ diagonal matrix whose $i$th diagonal entry is 1 if $i \in \{n+1-j\colon j \in T_G\}$ and 0 otherwise; by Theorem~\ref{lem:f(Av)} applied to $f = \partial_{\widetilde{\mathbf z} + \widetilde{\mathbf o}}\vol\,\mathcal F_{\widetilde G}(\widetilde{\mathbf x})$ and $A$ as above, the quadratic polynomial
\[
\partial_{\widetilde{\mathbf z} + \widetilde{\mathbf o}}\vol\,\mathcal F_{\widetilde G}(A\widetilde{\mathbf x})
\]
is Lorentzian and its Hessian has at most one positive eigenvalue. The rows and columns of this Hessian are naturally indexed by $\{n+1-j\colon j \in T_G\}$, and its $i,j$-th entry is the coefficient of $\frac{x_ix_j}{(e_i+e_j)!}$ in $\partial_{\widetilde{\mathbf z} + \widetilde{\mathbf o}}\vol\,\mathcal F_{\widetilde G}(A\widetilde{\mathbf x})$. This coefficient is $K_{G^r}(\widetilde{\mathbf z} + e_i + e_j)$. By Proposition~\ref{prop:entrywise-formula}, its Hessian is precisely $\widetilde K_{\ef{\mathbf d}}^-$.

We have thus shown that $\widetilde K_{\ef{\mathbf d}}^-$ has at most one positive eigenvalue, completing the proof.
\end{proof}

\section{On projections of polytopes in general}
\label{sec:proj}
 
 Recall the question stemming from Theorem \ref{thm:sigmaphi-sigmapsi-lorentzian}, as well as other examples mentioned in the Introduction: 
   
\begin{question1*}    {\it What conditions on the polytope/projection pair ascertain that the normalization of the  projection of the integer point transform of the polytope is Lorentzian?  }\end{question1*}  

Note that  $\varphi$ is a projection onto a coordinate hypersurface and the flow polytope $\F_{G}({\bf a})$ we are projecting lives in the nonnegative orthant. It is worth  noting that once we have a Lorentzian polynomial $f$ which equals the normalized projection onto a  coordinate hypersurface of an integer point transform of a polytope which belongs to the nonnegative orthant, then any derivative of $f$ is (1) Lorentzian, (2) also  the normalized projection onto a  coordinate hypersurface of an integer point transform of a polytope which belongs to the nonnegative orthant.   We formalize this observation here.
  
\begin{definition}
A polytope/projection pair $(P,\varphi)$ is said to be \textbf{admissible} if the polytope $P\subseteq\RR^m$ has vertices in $\ZZ^m$ and lives in the nonnegative orthant 
\[
H_+^m \overset{\rm def}= \{(x_1, \dots, x_m)\colon x_i \geq 0 \textup{ for all } i \in [m]\};
\]
we also require that $\varphi$ is a projection onto a coordinate $n$-dimensional hypersurface. Without loss of generality, we may assume $\varphi$ is projection onto the first $n$ components.
\end{definition}
Observe that $\varphi(P)\subseteq H_+^n$ lives inside the nonnegative orthant of $\RR^n$ and also has integral vertices.

To an admissible pair, we associate a polynomial $\sigma_P^\varphi$ obtained by projecting the integer point transform of $P$ according to $\varphi$; specifically,
\[
\sigma_P^\varphi(\mathbf x)\overset{\rm def}= \sum_{\mathbf p \in P\cap\ZZ^m} \mathbf x^{\varphi(\mathbf p)} = \sum_{\mathbf p \in \varphi(P)\cap\ZZ^n}(\varphi^{-1}(\mathbf p)\cap \ZZ^m)\mathbf x^{\mathbf p},
\]
where $\mathbf x = (x_1, \dots, x_n)$, and $\varphi^{-1}(\mathbf p)$ is interpreted as a subset of $P$. (Note that $\varphi(P)\subseteq H_+^n$ implies $\sigma_P^\varphi$ is actually a polynomial.)

\begin{proposition}
\label{prop:partial-f}
Let $f(x_1, \dots, x_n)$ be a Lorentzian polynomial so that $f = N(\sigma_P^\varphi)$ for some admissible pair $(P,\varphi)$. Then we have
\[
\frac\partial{\partial x_i}f = N(\sigma_{P_i}^\varphi), \hspace{0.4cm}\textup{ where }\hspace{0.4cm} P_i\overset{\rm def}= (P\cap H_{+i}^m) + \{-e_i\},
\]
where $H_{+i}^m = \{(x_1, \dots, x_m)\colon x_i \geq 1, \textup{ and } x_j\geq 0 \textup{ for all } j \in [m]\}$.
\end{proposition}

\begin{proof}
By Lemma~\ref{lem:NDN-is-coeff}, it suffices to show that if
\begin{equation}
\label{eqn:sigmaPi-from-sigmaP}
\sigma_P^\varphi(\mathbf x) = \sum_\alpha c_\alpha\mathbf x^\alpha,\hspace{0.4cm}\textup{ then } \hspace{0.4cm}\sigma_{P_i}^\varphi(\mathbf x) = \sum_{\alpha\colon \alpha_i \geq 1}c_\alpha\mathbf x^{\alpha - e_i}.
\end{equation}
Since $e_i \in\RR^n = \textup{im}\,\varphi$, we have $\varphi(P_i) = \varphi(P\cap H_{+i}^m) + \{-e_i\} = \varphi(P)\cap H_{+i}^n + \{-e_i\}$. A point $\beta \in \varphi(P_i)$ if and only if $\alpha \overset{\rm def}=\beta + e_i \in \varphi(P)\cap H_{+i}^n$. Furthermore, the fiber $\varphi^{-1}(\beta)\cap P_i$ is equal, up to translation by $e_i$, to the fiber $\varphi^{-1}(\alpha)\cap P$. Thus
\[
\sigma_{P_i}^\varphi (\mathbf x) = \sum_{\beta\in\varphi(P_i)\cap\ZZ^n}(\varphi^{-1}(\beta)\cap P_i \cap\ZZ^n)\mathbf x^\beta = \sum_{\substack{\alpha\in\varphi(P)\cap\ZZ^n\\\alpha_i \geq 1}} (\varphi^{-1}(\alpha)\cap P \cap\ZZ^n)\mathbf x^{\alpha - e_i}.
\]
Comparing the above expression to the definition of $\sigma_P^\varphi$, we have verified Equation~\eqref{eqn:sigmaPi-from-sigmaP} holds.
\end{proof}

\begin{remark}
\label{rem:varphi(Pi)}
We emphasize that the pair $(P_i,\varphi)$ is admissible when $(P,\varphi)$ is admissible. Furthermore, as discussed in the proof of Proposition~\ref{prop:partial-f},
\[
\varphi(P_i) = \varphi(P)\cap H_{+i}^n + \{-e_i\}.\qedhere
\]
\end{remark}

     \medskip
   
 We conclude by another intriguing question stemming from our work: \textit{which Lorentzian polynomials arise naturally as normalized  projections of   integer point transforms of   polytopes?}

\section*{Acknowledgments} We are grateful to Ricky I.\ Liu as well as Dave Anderson and Eugene Gorsky for inspiring conversations related to the idea of polytope/projection pairs. We are grateful to June Huh for inspiring conversations about Lorentzian polynomials as well as Louis Billera and Avery St.\ Dizier for inspiring conversation about polytopes in general and flow polytopes in particular. We are also grateful to Jacob Matherne for many motivating conversations on the above topics.

\end{document}